\documentclass{article}
\usepackage[utf8]{inputenc}

\usepackage{amssymb}
\usepackage{amsthm}
\usepackage{amsmath,amscd}
\usepackage[mathscr]{euscript}
\usepackage[all]{xy}
\usepackage{lmodern}
\usepackage[T1]{fontenc}
\usepackage[textwidth=14cm,hcentering]{geometry}
\usepackage[colorlinks=true,linkcolor=red,citecolor=blue]{hyperref}
\usepackage{cleveref}
\usepackage{mathtools}
\usepackage{xspace}
\usepackage{tikz}
\usetikzlibrary{cd}
\usetikzlibrary{arrows,positioning}
\usepackage{bbm}

\title{Localizing motives of Azumaya algebras}
\author{Maxime Ramzi }
\date{}

\newtheorem{thm}{Theorem}[section]
\newtheorem{lm}[thm]{Lemma}
\newtheorem{prop}[thm]{Proposition}
\newtheorem{cor}[thm]{Corollary}

\newtheorem*{thm*}{Theorem}
\newtheorem*{cor*}{Corollary}

\theoremstyle{definition}
\newtheorem{defn}[thm]{Definition}
\newtheorem{cons}[thm]{Construction}
\newtheorem{assu}[thm]{Assumption}
\newtheorem{nota}[thm]{Notation}

\newtheorem{ex}[thm]{Example}
\newtheorem{exs}[thm]{Examples}
\newtheorem{exc}[thm]{Exercise}
\newtheorem{rmk}[thm]{Remark}
\newtheorem{ques}[thm]{Question}
\newtheorem{conj}[thm]{Conjecture}
\newtheorem{warn}[thm]{Warning}
\newtheorem{obs}[thm]{Observation}
\newtheorem*{ques*}{Question}
\newtheorem*{rmk*}{Remark}
\newtheorem*{ex*}{Example}

\AtBeginEnvironment{defn}{\pushQED{\qed}}
\AtEndEnvironment{defn}{\popQED\enddefn}
\AtBeginEnvironment{defn*}{\pushQED{\qed}}
\AtEndEnvironment{defn*}{\popQED\enddefn}
\AtBeginEnvironment{cons}{\pushQED{\qed}}
\AtEndEnvironment{cons}{\popQED\endcons}
\AtBeginEnvironment{assu}{\pushQED{\qed}}
\AtEndEnvironment{assu}{\popQED\endassu}
\AtBeginEnvironment{nota}{\pushQED{\qed}}
\AtEndEnvironment{nota}{\popQED\endnota}
\AtBeginEnvironment{nota*}{\pushQED{\qed}}
\AtEndEnvironment{nota*}{\popQED\endnota}
\AtBeginEnvironment{conv}{\pushQED{\qed}}
\AtEndEnvironment{conv}{\popQED\enddefn}\AtBeginEnvironment{ex}{\pushQED{\qed}}
\AtEndEnvironment{ex}{\popQED\endex}
\AtBeginEnvironment{exs}{\pushQED{\qed}}
\AtEndEnvironment{exs}{\popQED\endexs}
\AtBeginEnvironment{exc}{\pushQED{\qed}}
\AtEndEnvironment{exc}{\popQED\endexc}
\AtBeginEnvironment{rmk}{\pushQED{\qed}}
\AtEndEnvironment{rmk}{\popQED\endrmk}
\AtBeginEnvironment{rmk*}{\pushQED{\qed}}
\AtEndEnvironment{rmk*}{\popQED\endrmk}
\AtBeginEnvironment{ques}{\pushQED{\qed}}
\AtEndEnvironment{ques}{\popQED\endques}
\AtBeginEnvironment{ques*}{\pushQED{\qed}}
\AtEndEnvironment{ques*}{\popQED\endques}
\AtBeginEnvironment{conj}{\pushQED{\qed}}
\AtEndEnvironment{conj}{\popQED\endconj}
\AtBeginEnvironment{warn}{\pushQED{\qed}}
\AtEndEnvironment{warn}{\popQED\endwarn}
\AtBeginEnvironment{obs}{\pushQED{\qed}}
\AtEndEnvironment{obs}{\popQED\endobs}

\newcommand{\op}{^{\mathrm{op}}}
\newcommand{\cat}{\mathbf}

\newcommand{\Catperf}{\mathrm{Cat}^{\mathrm{perf}}}
\newcommand{\on}{\operatorname}

\newcommand{\Map}{\on{Map}}

\newcommand{\Z}{\mathbb Z}

\newcommand{\Sp}{\cat{Sp}}

\renewcommand{\O}{\mathcal{O}}

\newcommand{\PrL}{\mathrm{Pr}^\mathrm{L} }

\newcommand{\Alg}{\mathrm{Alg}}
\newcommand{\CAlg}{\mathrm{CAlg}}

\newcommand{\Mod}{\cat{Mod}}

\newcommand{\HH}{\mathrm{HH}}
\newcommand{\THH}{\mathrm{THH}}
\newcommand{\Perf}{\mathrm{Perf}}

\newcommand{\pt}{\mathrm{pt}}

\newcommand{\Br}{\mathrm{Br}}
\newcommand{\Pic}{\mathrm{Pic}}
\newcommand{\BBr}{\mathbf{Br}}
\newcommand{\PPic}{\mathbf{Pic}}
\newcommand{\Gm}{\mathbb{G}_m}

\newcommand{\et}{\mathrm{\'et}}
\newcommand{\der}{\mathrm{der}}
\newcommand{\Mot}{\mathrm{Mot}}
\newcommand{\Sh}{\mathrm{Sh}}
\newcommand{\udl}[1]{\underline{#1}}
\newcommand{\QCoh}{\mathrm{QCoh}}
\newcommand{\Aut}{\mathrm{Aut}}
\newcommand{\U}{\mathcal{U}_{\mathrm{loc}}}
\newcommand{\Spec}{\mathrm{Spec}}
\newcommand{\one}{\mathbbm{1}}

\newcommand{\st}{\mathrm{st}}
\newcommand{\dbl}{\mathrm{dbl}}

\newcommand{\Sch}{\mathrm{Sch}}
\newcommand{\qcqs}{{\mathrm{qcqs}}}

\begin{document}

\maketitle
\begin{abstract}
    We prove that the natural map from the derived Brauer group of a qcqs scheme $X$ to the Picard group of $X$-linear motives is injective, extending results of Tabuada and Tabuada-Van den Bergh.
\end{abstract}

\section*{Introduction}
The Brauer group of a scheme $X$, introduced in this generality in \cite{GrBrI}, is a natural invariant of $X$ which, on the one hand, carries étale cohomological information via its description in terms of $H^2_\et(X,\Gm)$, and on the other hand is describable in purely quasi-coherent terms, in a similar fashion to the Picard group of $X$. 

From the latter perspective, the Brauer group of a scheme $X$ is a \emph{categorification} of the Picard group, and in fact, if we also include \emph{derived} Azumaya algebras to get the derived Brauer group, as introduced in \cite{Toen}, it can be defined as some kind of higher Picard group, namely, the Picard group of some category\footnote{Really, $\infty$-category.} of ``quasi-coherent stacks on $X$'', or ``$X$-linear stable $\infty$-categories''. 

On the other hand, there is a natural \emph{de}categorification process for $X$-linear categories, given by $K$-theory. This can be refined by considering, instead of $K$-theory, the universal functor with certain properties that resemble those of $K$-theory, called the \emph{universal localizing invariant}, $\U^X$. This is a symmetric monoidal functor\footnote{As opposed to $K$-theory, which is only lax symmetric monoidal.} $\U^X: \Catperf_X\to \Mot_X$, and thus it carries our invertible $X$-linear categories to invertible ($X$-linear, noncommutative) \emph{motives}. The following natural question arose rather early in the study of localizing motives, as developed by Tabuada:
\begin{ques*}\label{ques:inj}
Is the induced morphism $\Br^\der(X)\to \Pic(\Mot_X)$ injective ? How far is it from being surjective ?
\end{ques*}
Though the latter part of the question appears difficult for now, in \cite[Theorem 9.1]{Tab0}, Tabuada showed that the answer to the first part was yes for fields, namely the map $\Br^\der(k)\to \Pic(\Mot_k)$ is injective.  Later, Tabuada and Van den Bergh made some further progress on the first question, also providing interesting variants thereof. 

We first summarize their joint results about this question as follows: 
\begin{thm*}[{\cite[Theorem 2.1]{TvdB1},\cite[Theorem B.15]{TvdB2}}]

    Let $X$ be a connected\footnote{We only assume this for simplicity of writing.} qcqs scheme. 

    \begin{itemize}
        \item  Suppose $A$ is a classical Azumaya algebra of rank $r$. There is an equivalence of motives $\U^X(A)[\frac{1}{r}]\simeq \U^X(\O_X)[\frac{1}{r}]$; 
        \item  If $X$ is affine and noetherian\footnote{Tabuada and Van den Bergh already knew that the noetherian assumption could likely be dropped. Further, by using a result of de Jong--Gabber instead of the result of Gabber they use, it can be extended to schemes with an ample line bundle.}, then the map $A\mapsto \U^X(A)_\mathbb Q$ distinguishes classical and derived Morita classes of classical and derived Azumaya algebras - that is, if $A$ is a derived Azumaya algebra whose rational motive is equivalent to that of a classical Azumaya algebra, then $A$ is Morita equivalent to a classical Azumaya algebra. 
    \end{itemize}
\end{thm*}
In fact, Tabuada went further in the direction of injectivity. Indeed, he considers a reasonable stable subcategory $\Mot_X^{\mathrm{sm.p}}\subset \Mot_X$ containing among other things all motives of Azumaya algebras, but also all motives of smooth proper schemes over $X$, and consider the image of the motives of Azumaya algebras in $K_0(\Mot_X^{\mathrm{sm.p}})$ which is (almost\footnote{In fact, Tabuada proves something about \emph{additive} or \emph{splitting} motives. For fields, and restricting to motives of smooth and proper categories, the functor from additive to localizing motives is conjecturally an equivalence.}) Toën's secondary $K$-theory group $K_0^{(2)}(X)$. By essentially reducing to fields and performing an explicit analysis there, Tabuada proves:
\begin{thm*}[{\cite[Theorem 1.3]{Tab2}}]
    Let $X$ be a regular, integral qcqs scheme over a field. The natural map $\Br^\der(X)\to K_0^{(2)}(X)$\footnote{or more properly, $GL_1(K_0^{(2)}(X))$.} from the derived Brauer group of $X$ is injective. 
\end{thm*}
\begin{rmk*}
    The reduction to fields actually works more generally: for any regular separated noetherian scheme $X$, restriction to generic points induces an injection on Brauer groups. In any case, in both situations the key special case is that of fields. 
\end{rmk*}
In this work and its prequel \cite{Tab1}, Tabuada also draws more precise conclusions and extends this result a bit beyond the case described here. In particular, the above theorem certainly guarantees that for such $X$, the map $\Br^\der(X)\to \Pic(\Mot_X)$ is injective (and is in fact stronger). 

Our goal in this short paper and its companion \cite{companion} is to pursue this line of investigation. Our basic insight in both papers is the same, though the technicalities involved in pushing it through are of drastically different nature, and in particular the present paper is to be thought of as much softer and more elementary than the companion. The reader should feel free to read it as a light introduction to the latter. 

In this paper, we concern ourselves with the injectivity question mentioned earlier, and do not attempt to say anything about secondary $K$-theory, or about surjectivity. We essentially observe that the natural map $\Br^\der(X)\to \Pic(\Mot_X)$ is, on a large subgroup of $\Br^\der(X)$, split by a natural and simple construction: that of the \emph{universal determinant} for the Azumaya algebra $A$, which we observe only depends on the (noncommutative, or localizing) motive of $A$. This method already buys us injectivity on a large subgroup of the derived Brauer group of $X$, namely $H^2_\et(X,\Gm)$ (in fact, it buys us a bit more than this, but this is too technical for the purposes of this introduction).  

However, if we combine our determinant approach with (a variation of) the aforementioned \cite[Theorem B.15]{TvdB2} to obtain a full injectivity statement: 
\begin{thm*}
    Let $X$ be a qcqs scheme. The map $$\Br^\der(X)\to\Pic(\Mot_X)$$ is injective. 
\end{thm*}
\begin{rmk*}
    The following can also be deduced from the work of Tabuada, but we point it out explicitly nonetheless: the above theorem implies that if $A$ is a classical, non-$2$-torsion Azumaya algebra, $\U^X(A)\not\simeq \U^X(A\op)$. This is interesting since many localizing invariants are $(-)\op$-invariant, such as $K$-theory and the usual variants of $\THH/\HH_k$. 
\end{rmk*}

Our method of proof is different from that of \cite{Tab1}, and instead relies on a local analysis of the map $\Br^\der(X)\to \Pic(\Mot_X)$. To make this local analysis, it is crucial to view $\Br^\der$ as only $\pi_0$ of a more refined Brauer \emph{space} or spectrum, $\BBr^\der$. This spectrum also contains the information of $\Pic(X)$ and $\Gm(X)$, and these are glued together in some possibly interesting way, but more importantly for us, the results of Toën \cite{Toen} guarantee that $\BBr^\der$ can be étale-locally completely understood from $\PPic$.  Similarly, $\Pic(\Mot_X)$ is only $\pi_0$ of the Picard space or spectrum, $\PPic(\Mot_X)$.

The $H^2_\et(X,\Gm)$ part of the derived Brauer group arises from a canonical map $$B^2\Gm\to \BBr^\der$$ of presheaves of spectra, which is not far from an equivalence upon (étale) sheafification of the source, and this allows us to understand the local behaviour of the map $\BBr^\der(-)\to~\PPic(\Mot_{-})$ from the behaviour of the much simpler map $B^2\Gm(-)\to \PPic(\Mot_{-})$, for which we essentially construct a retraction. 
\begin{ex*}
Even in the case of a field, our approach is different from Tabuada's. This special case is particularly interesting because of its elementary nature: our method boils down to the observation that for a central simple algebra $D$ over a field $k$, one can recover the Morita equivalence class of $D$ from the knowledge of $K_{\leq 1}(k^\mathrm{sep}\otimes_k D)$ together with its $\mathrm{Gal}(k^\mathrm{sep}/k)$-action. 
\end{ex*}

\subsection*{Related work}
As already mentioned in the body of the introduction, this work follows the investigations of \cite{TvdB1,TvdB2,Tab0,Tab1,Tab2} and owes them a clear intellectual debt. The reader is invited to go through these papers for their comprehensive lists of examples and further detailed information. 

Our approach using the twisted determinant of an Azumaya algebra is closely related to the work of Kahn-Levine \cite{kahnlevine}, and in fact our $\PPic_A$ is a less-truncated version of their $\mathcal K_0^A$ or $\mathbb Z_A$ (though we should be careful about relating their $\mathcal K_1^A$ and the fiber from \Cref{rmk:otherapproach}). As pointed out to us by Ben Antieau (and explained in \Cref{rmk:otherapproach}), (part of) our $\PPic_A(X)$ also has a precedent in terms of Severi--Brauer varieties when $A$ is a classical Azumaya algebra, as the subgroupoid $\PPic^{\heartsuit,1}(P)$ of line bundles of degree $1$ on the Severi--Brauer scheme associated to $A$. 

Finally, let us discuss briefly the contents of the companion paper \cite{companion}, where we focus on the $p$-primary torsion part of the Brauer group, for some prime $p$ invertible in the base. Our approach there is similar in the sense that it builds on an investigation of the local behaviour of certain Brauer and Picard (pre)sheaves. However, there we exploit a potentially more destructive operation on Azumaya algebras: we consider their $K(1)$-local $K$-theory. On the one hand, this is a more destructive operation, but on the other hand, this defines an exact functor on the category of motives, and is therefore more likely to be able to say anything about secondary $K$-theory. Our results there are of slightly different nature, since generally speaking, taking the bare spectrum ``$L_{K(1)}K(A)$'' \emph{does} lose too much information - what we observe there is that remembering a reasonable amount of extra information actually is sufficient to recover the Brauer class of $A$ under some assumptions. 

\subsection*{Conventions}
We use the language of $\infty$-categories as extensively developed in \cite{HTT,HA}, though we drop the ``$\infty$-'' and simply refer to them as categories, and say ``$1$-category'' for ordinary categories when the need arises.

Our Azumaya algebras and quasicoherent sheaves are typically implicitly derived, so that $\QCoh,\PPic$ etc. by default refer to the derived versions of quasicoherent sheaves and line bundles. However our schemes are classical and qcqs.   
\subsection*{Acknowledgements}
I am grateful to Ben Antieau for his enthusiasm regarding this project, and a number of helpful comments (including an initial discussion about a year ago which sparked my interest in this question). I also had interesting conversations related to this project with Mura Yakerson and Denis-Charles Cisinski.

This paper owes an obvious intellectual debt to the work of Tabuada on motives; it was also sparked by an inspiring talk given by Tabuada in the conference ``Recent developments in algebraic $K$-theory'' organized at the University of Warwick, and I am grateful to the organizers, Rudradip Biswas and Marco Schlichting, for putting together this event. 

This research was funded by the Deutsche Forschungsgemeinschaft (DFG, German Research Foundation) – Project-ID 427320536 – SFB 1442, as well as by Germany’s Excellence Strategy EXC 2044 390685587, Mathematics Münster: Dynamics–Geometry–Structure. 
\section{Preliminaries}
In this short section, we gather a few definitions and facts about (noncommutative or localizing) motives, and about Brauer and Picard sheaves. 
\subsection{$X$-linear motives}
\begin{defn}
    Let $X$ be a qcqs scheme. We let $\Catperf_X$ denote $\Mod_{\Perf(X)}(\Catperf)$, the category of modules over $\Perf(X)$ in $\Catperf$, the category of stable categories. We also call them ``$X$-linear categories''. 

    By \cite{BGT} in the absolute case, or for example \cite[Section 5]{HSS} and \cite[Appendix B]{RSW} in the relative case, there is an associated stable, presentably symmetric monoidal category $\Mot_X$ of $X$-linear (localizing\footnote{We shall not consider any other kind of motive in this paper, and therefore drop the word ``localizing'' throughout.}) motives. We let $\U^X: \Catperf_X \to \Mot_X$ denote the associated universal finitary localizing invariant. 
\end{defn}
\begin{nota}
We shall often abuse notation and, for a quasicoherent derived algebra $A\in\Alg(\QCoh(X))$, write $\U^X(A)$ in place of $\U^X(\Perf(A))$.  
\end{nota}
Recall that by \cite[Theorem 5.15]{HSS}, (connective) $K$-theory, for $X$-linear categories, can be described as the mapping space $\Map_{\Mot_X}(\U^X(\Perf(X)), \U^X(-))$ - the only thing we will truly need to know here is that (connective) $K$-theory of an $X$-linear category depends functorially on its motive\footnote{The precise representability result uses that all objects in $\Perf(X)$ are dualizable, whereas the latter fact uses essentially nothing.}. 

More generally, any motive $M\in\Mot_X$ induces a (finitary) localizing invariant $K_M:= K(M\otimes_{\Perf(X)}-)$, which in turn induces a Nisnevich sheaf $$U\mapsto K_M(\Perf(U))$$ on $\Sch^\qcqs_{/X}$. 
\begin{nota}
    We let $\mathcal K_M$ denote the sheaf described above. 
\end{nota}

If we consider connective $K$-theory, this is a Nisnevich sheaf of spaces or $\mathbb E_\infty$-groups. 
\subsection{Brauer and Picard sheaves}
\begin{nota}
For a symmetric monoidal category $C$, we let $\PPic(C)$ denote the full subspace of $C^\simeq$ spanned by $\otimes$-invertible objects, and $\Pic(C):= \pi_0\PPic(C)$. 

If $X$ is a qcqs scheme, we let $\PPic(X):=\PPic(\Perf(X))$ and similarly for $\Pic$. 

We let $\BBr(X) := \PPic(\Catperf_X)$ and $\Br(X) := \pi_0\BBr(X)$. By \cite{Toen}\footnote{See also \cite{AntieauGepner} for the variant where $X$ is a connective spectral scheme.}, this is exactly the group of derived Azumaya algebras over $X$ up to Morita equivalence. 
\end{nota}
\begin{rmk}
    From now on, all our Brauer groups/Azumaya algebras and line bundles are implicitly derived unless explicitly stated, so we removed it from the notation. 
\end{rmk}
\begin{rmk}
We have a canonical equivalence $\Omega(\BBr(X),\Perf(X))\simeq  \Aut_{\Perf(X)}(\Perf(X)) \simeq \PPic(X)$, as well as a similar canonical equivalence $\Omega(\PPic(X),\O_X)\simeq \Gm(X)$. Thus $\BBr(X)$ is a $2$-truncated spectrum, with $\pi_0$ given by the Brauer group, $\pi_1$ the Picard group, and $\pi_2$ the group of units, each of which being a categorification of the next.  
\end{rmk}
The main results of \cite{Toen} are that $U\mapsto \BBr(U)$ is an étale sheaf of spectra, and that any (derived) Azumaya algebra is étale locally trivial, and hence the map of sheaves $\pt\to \BBr(-)$ classifying the trivial Azumaya algebra is an effective epimorphism of étale sheaves on any qcqs scheme $X$. It follows that, as étale sheaves, we have an equivalence $$B_\et\PPic\simeq \BBr$$ 

Now any classical line bundle is Zariski and hence étale locally trivial, but $\Pic(X)$ also contains suspensions of the unit, which are not. It follows by considering stalks that the sheafification of $U\mapsto \Pic(U)$ is $\udl{\Z}$, the constant sheaf with value $\Z$. The truncation map $\PPic\to \Pic$ thus induces a map of sheaves $\BBr\simeq B_\et\PPic\to B_\et\udl\Z$.

\begin{nota}
    We let $r: \Br(X)\to H^1_\et(X;\Z)$ denote the map induced on $\pi_0$ of global sections.
\end{nota}
\section{Twisted determinants}
In this section, we prove our main result by introducing, for a derived Azumaya algebra $A$, a twisted form of the Picard sheaf $\PPic$, which we denote $\PPic_A$. 

\begin{nota}
    Throughout this section, we will fix a qcqs scheme $X$. All (pre)sheaves and related notions (such as sheafification) are to be interpreted with respect to $X_\et$, the small étale site of $X$. 
\end{nota}

To motivate our definition, recall the following: 
\begin{prop}\label{prop:Pic=K}
   The determinant map $\det: K_{[0,1]}(U)\to\PPic(U)$, natural in $U$, induces an equivalence $L_\et K_{[0,1]}\to \PPic$ of étale sheaves of connective spectra on $X$.
\end{prop}
Here, $L_\et$ denotes étale sheafification (in connective spectra) and $K_{[0,1]}$ denotes the Postnikov section of $K$ in degrees $[0,1]$. 

Before proving this proposition, let us recall what we mean here by ``determinant map''. In \cite[Chapter I]{Det}, Knudsen and Mumford construct, for a commutative ring $R$ with connected $\Spec(R)$ (we only say this for simplicity), a symmetric monoidal functor $\det: (\mathrm{Proj}_R^\simeq,\oplus)\to \PPic(\mathrm{GrProj}_R)$, where $\mathrm{GrProj}_R$ is the ($1$-)category of graded projective $R$-modules with the usual graded tensor product and the Koszul sign rule, which sends a projective module $P$ to $\Lambda^{\mathrm{rk}(P)} P[\mathrm{rk}(P)]$ (see \textit{loc. cit.} for how to deal with the case where $\Spec(R)$ is not connected), where $M[n]$ means $M$ in grading $n$.  Now there is a strong symmetric monoidal functor of $1$-categories $\mathrm{GrProj}_R\to \mathrm{Ch}(\mathrm{Proj}_R)$ interpreting graded projective $R$-modules as chain complexes with trivial differential, and finally a symmetric monoidal functor $\mathrm{Ch}(\mathrm{Proj}_R)\to D(R)$ from the $1$-category of chain complexes of projective $R$-modules to the derived $\infty$-category of $R$. Together, these compose to a functor $\det: (\mathrm{Proj}_R^\simeq,\oplus)\to \PPic(R)$. 

Since the target is a grouplike $\mathbb E_\infty$-monoid, this factors through the group completion of the source, which is $K(R)$, and altogether we get a map $\tau_{[0,1]}K(R)\to \PPic(R)$, which we dub ``the determinant''. Since $\PPic$ is a sheaf on schemes, this also globalizes to schemes\footnote{By \cite[Example A.0.6]{EHKSY}, connective $K$-theory of animated commutative rings is left Kan extended from smooth rings, and thus there is also a determinant for animated rings and more generally derived schemes; but it is not expected to extend to $\mathbb E_\infty$-$\mathbb Z$-algebras in any reasonable way.}, and this is what the above proposition refers to. 

\begin{proof}
Since $\PPic$ is already a sheaf for the étale topology, it suffices to prove the same claim for the Zariski sheafification of $K_{[0,1]}$.

It therefore suffices to check on stalks for the Zariski topology, that is, local rings. Since both $K$-theory and the Picard space are finitary, this means checking that for a local commutative ring $R$, $\det_R : K_{[0,1]}(R)\to \PPic(R)$ is an equivalence. This is classical: for $\pi_0$ it follows from the fact that every projective module over a local ring is free, and for $\pi_1$ it follows from the fact that $\det_R : K_1(R)\cong R^\times$ for local rings (see e.g. \cite[Lemma III.1.4]{weibel}. 
\end{proof}
Consequently, $\PPic$ admits the structure of a sheaf of commutative ring spectra, and we will implicitly consider it as such.

With this in mind, we introduce: 
\begin{defn}
    Let $C$ be an $X$-linear category, i.e. a $\Perf(X)$-module in $\Catperf$. We let $\PPic_C$ denote the étale sheafification of $$U\mapsto \tau_{[0,1]}K(C\otimes_{\Perf(X)}\Perf(U))$$
    i.e. the étale sheafified $1$-truncation of $\mathcal K_C$. 
    It has a canonical $\PPic$-module structure in étale sheaves of connective spectra on $X$. 

    More generally, $C\mapsto \PPic_C$ extends to a functor\footnote{The reader should note that this is \emph{not} an exact functor.} $\Mot_X\to \Mod_{\PPic}(\Sh_\et(X; \Sp_{\geq 0}))$. 

    If $A$ is a sheaf of (derived) algebras over $X$, we let $\PPic_A$ abbreviate $\PPic_{\Perf(A)}$.
\end{defn}
\begin{nota}
    For $C$ an $X$-linear motive, we let $\Pic_C$ denote $\pi_0^\et(\PPic_C)$, that is, the étale sheafified $\pi_0$ of $\PPic_C$. 
\end{nota}
Our goal is to show that one can in fact (almost) recover $A$ (up to Morita equivalence) from $\PPic_A$, when $A$ is an Azumaya algebra. To prove this, we will interpret twisted forms of $\PPic$ in cohomological terms. 

Before doing so, we need to identify what happens locally. 
\begin{lm}\label{lm:formofPic}
    Let $A$ be an Azumaya algebra over $X$. $\PPic_A$ is étale-locally equivalent to $\PPic$ as a sheaf of $\PPic$-module spectra. 
\end{lm}
\begin{proof}
    This follows from the fact that $\Perf(A)$ is étale locally equivalent to $\Perf(X)$, by \cite[Proposition 1.14]{Toen}. 
\end{proof}
In particular, $\Pic_A := \pi^\et_0(\PPic_A)$ is étale-locally $\pi^\et_0(\PPic)$ which is the constant sheaf $\udl{\mathbb Z}$ by \Cref{prop:Pic=K}. This sheaf is the étale-sheafified analogue of $\mathcal K_0^A$ from \cite{kahnlevine}.

\begin{cons}\label{cons:e}
    For $U$ a scheme, the full subgroupoid of $\Catperf_U$ spanned by $\Perf(U)$ is equivalent to $B\PPic(U)$.The restriction of the $K$-theory functor $\Catperf_U\to \Mod_{K(U)}$ to this full subgroupoid is strong symmetric monoidal and natural in $U$. Thus we get a natural map of spectra $$\PPic(U)\to K(U)^\times$$ and so, by truncating and sheafifying on the étale site $X_\et$, a map of sheaves of connective spectra $$e:\PPic\to \PPic^\times$$ 
\end{cons}
The relevance of the map $e$ is as follows: it induces by delooping a map $$B_\et\PPic\to B_\et\PPic^\times=B_\et\Aut_{\PPic}(\PPic)$$ and hence, by local triviality of the Brauer group, a map from $\BBr$ to the space of twisted forms of $\PPic$ (i.e. $\PPic$-module sheaves locally equivalent to $\PPic$). On global sections, this map is exactly the map that sends an Azumaya algebra $A$ to $\PPic_A$. 

\begin{lm}\label{lm:effectofe}
    On $\pi_0$, the map $e$ fits into the following commutative diagram, with the reduction mod $2$ map $\udl\Z\to\udl{\Z/2}$, \[\begin{tikzcd}
	{\udl{\Z}} & {\udl{\Z/2}\cong\Z^\times} \\
	\Pic & {\Pic^\times}
	\arrow[from=1-1, to=1-2]
	\arrow[from=1-1, to=2-1]
	\arrow[from=1-2, to=2-2]
	\arrow["e"', from=2-1, to=2-2]
\end{tikzcd}\]
where the vertical maps are isomorphisms upon sheafification.

Taking the fiber over $0\in\mathbb Z/2$ in the target and $1\in\mathbb Z$ in the source of $e$,  the induced map $B\Gm\to B\Gm$ is the identity.     
\end{lm}
\begin{proof}
For the first part, note that $\pi_0 K$ is locally generated by $\O$, whose determinant is by design $\Sigma \O$, whose $K$-theory class is $-1 \in \Z^\times \in K_0^\times$. 

For the second part, note that the map $\PPic\to \PPic^\times$ sends an invertible object $\mathcal L$ to its $K$-theory class $[\mathcal L]$, and an automorphism of $\mathcal L$ to the induced loop in $\Omega(K(U),[\mathcal L])$. This can for example be checked by noting that there is a map $C^\simeq\to K(C)$ natural in the stable category $C$. 

Thus, at the object $\O\in \PPic$ it induces on $\pi_1$ the canonical map $\Gm\to K_1$ which we mentioned in \Cref{prop:Pic=K} is an isomorphism upon sheafification (it's a left inverse of \emph{the} isomorphism by which we identify $L_\et K_1$ with $\Gm$). 
\end{proof}

We can now identify more precisely $\PPic_A$ as a twisted form of $\PPic$. For this, we first need to understand its $\pi_0^\et$. 
\begin{lm}
\'Etale sheaves of abelian groups that are locally isomorphic to $\udl{\mathbb Z}$ are classified by $H^1_\et(X;\Z^\times)=H^1_\et(X;\Z/2)$. 

Furthermore, under this identification for an Azumaya algebra $A$, the twist of $\udl{\Z}$ given by $\Pic_A$ is the image in $H^1_\et(X;\Z/2)$ of the class of $A$ under the composition $$\Br(X)\xrightarrow{r}H^1_\et(X;\Z)\to H^1_\et(X;\Z/2)$$. 
\end{lm}
\begin{proof}
The first part is classical and follows from the fact that the groupoid of abelian groups abstractly isomorphic to $\Z$ forms a $B(\Z^\times)$. 

For the second part, we note that the map $r:\Br(X)\to H^1_\et(X;\Z)$ is given by taking global sections on the map $B_\et\PPic\to B_\et \Pic$. Since we know by \Cref{lm:effectofe} that $B\Pic\to B\Pic^\times$ is identified with reduction mod $2$, $\udl{\Z}\to\udl{\Z/2}$ upon sheafification, the claim follows. 
\end{proof}

We already obtain: 
\begin{cor}
Let $A$ be an Azumaya algebra. There exists an isomorphism $$\Pic_A\cong L_\et \Pic (\cong \udl\Z)$$ of sheaves of abelian groups if and only the image of $[A]$ under $$\Br(X)\xrightarrow{r} H^1_\et(X;\Z)\to H^1_\et(X;\Z/2)$$ is zero. 

In particular, if $\U^X(A)\simeq \U^X(\O_X)$, then $r([A]) $ is divisible by $2$. 
\end{cor}
Now we deal with the more interesting part, namely $\pi_1$. 
\begin{defn}
    Let $2\PPic$ denote the pullback $\PPic\times_\Z 2\Z$. 
\end{defn}
As explained after \Cref{cons:e}, the assignment $A\mapsto \PPic_A$ corresponds to the map $Be: B_\et\PPic\to B_\et\PPic^\times$. Taking fibers over the map $B_\et\PPic^\times\to B_\et\udl{\Z/2}$ we get a map $B_\et2\PPic\to B^2_\et\Gm$, as well as a (definitional) map $B_\et 2\PPic\to B_\et 2\Z$. 
\begin{lm}
    The above maps induce an equivalence $B_\et 2\PPic\simeq B_\et 2\Z\times B^2_\et\Gm$. 
\end{lm}
\begin{proof}
On $\pi_1$ this is clear by design of $B2\PPic$ and the map $B2\Z$. So we are left with checking that $B\PPic\to B\PPic^\times$ induces an isomorphism on $\pi_2$, but this is simply the second part of \Cref{lm:effectofe}. 
\end{proof}
\begin{rmk}
    As spaces, it is already true that $B_\et \PPic\simeq B_\et\Z\times B^2_\et\Gm$, but this equivalence does not hold as groups. On the other hand, the above equivalence is an equivalence of $\mathbb E_\infty$-groups. 
\end{rmk}
\begin{rmk}
    Note that unlike the equivalence $B_\et\PPic\simeq B_\et\Z\times B^2_\et\Gm$ from \cite{Toen}, in the above lemma the map $B_\et2\PPic\to B_\et^2\Gm$ essentially \emph{a priori} classifies (a variant of) the construction $A\mapsto \PPic_A$. So this lemma is in some sense the key connection between the description of $\Br(X)$ in terms of $H^2_\et(X,\Gm)$ and the construction $A\mapsto \PPic_A$. 

    The proof of \Cref{lm:effectofe} shows that, in fact, the identification of $\pi_2$ with $\Gm$ is the same for both equivalences. 
\end{rmk}
We are now ready to prove a weaker version of our main theorem - we state more than what we need later, to indicate exactly how far this method can be pushed: 
\begin{thm}
    Let $X$ be a qcqs scheme and let $A$ be an Azumaya algebra over $X$. Suppose $\U^X(A)$ is equivalent the unit of $\Mot_X$.

    Then the image of $A$ under $\Br(X)\to H^1_\et(X;\Z/2)$ is $0$. This guarantees that $A$ admits a lift along $$\pi_0\Gamma_\et(X,B2\PPic)\to\pi_0\Gamma_\et(X,B\PPic)\cong \Br^\der(X) $$

    For any such lift $\tilde A$, its image in $H^2_\et(X,\Gm)$ under the canonical isomorphism $$\pi_0\Gamma_\et(X,B2\PPic)\cong H^1_\et(X;2\Z)\times H^2_\et(X;\Gm)$$ is $0$. 
\end{thm}
\begin{rmk}
  This method based on local analysis does not seem to be able to distinguish $2H^1_\et(X;\Z)$ from $0$ in general, essentially because $\Sigma^2$ acts trivially on $\U^X(X)$. As we explain below, we can distinguish it if we combine that method with the ones from \cite{TvdB2}.
\end{rmk}
\begin{proof}
    We have already seen the first part. For the second part, note that if $\U^X(A)\simeq \U^X(\O_X)$, then the image of $A$ under the map $B_\et\PPic\to B_\et \PPic^\times$ is trivial so we can choose the zero lift to $B^2\Gm$, which gives an example of \emph{a} lift for which the image is $0$. 

    Any other lift corresponds to a second nullhomotopy of the corresponding point in $B\Z/2$, i.e. to a point in $H^0_\et(X,\Z/2)$. Thus given any lift, its image in $H^2_\et(X,\Gm)$ is in the image of the canonical map $H^0_\et(X,\Z/2)\to H^2_\et(X,\Gm)$ induced by the fiber sequence $\Z/2\to B^2\Gm\to B\PPic^\times$. Thus we have to prove that this map is $0$.

    But any class in $H^0_\et(X,\Z/2)$ is pulled back from a class in $H^0_\et(\coprod_i \Spec(\Z),\Z/2)$ for some finite (since $X$ is qc) indexing set $I$, and $H^2_\et(\coprod_i \Spec(\Z),\Gm)=0$, so this is clear.
\end{proof}
\begin{rmk}\label{rmk:otherapproach}
    Let us comment on a different approach to the same proof. For an Azumaya algebra $A$ over $X$, fixing an isomorphism (if it exists!) $\Pic_A\cong\udl\Z$, one can then consider the fiber over $1\in\udl\Z$ of $\PPic_A\to \Pic_A$. This is ``the $\Gm$-part'' of $\PPic_A$ and is where the interesting information lies: this is a twisted form of $B\Gm$. 

    One can in fact prove that $H_1$ of this twisted form is canonically $\Gm$, and hence this twisted form lives in the (cohomology of) fiber of the canonical map $B\Aut(B\Gm)\to B\Aut(\Gm)$, which is none other than $B^2\Gm$, and one can then verify that the corresponding cohomology class is \emph{also} the expected one in $H^2_\et(X,\Gm)\subset \Br(X)$. The details are more involved to verify, so we went for the above approach instead.
    
    We note, on the other hand, that when $A$ is classical, this fiber has a natural geometric interpretation which was pointed out to us by Ben Antieau. Indeed, in this case, $A$ has an associated Severi-Brauer variety $P\to X$ which is a twisted form of $\mathbb P^n_X$ for some $n$. The above fiber is then isomorphic to $(U\mapsto \PPic^{\heartsuit,1}(P_{U}))$, the sheaf that associates to an étale map $U\to X$ the groupoid of (classical) line bundles on $P\times_X U$ of degree $1$. This follows from considering the semi-orthogonal decomposition $\Perf(P)= \langle \Perf(X), \Perf(A),...,\Perf(A^{\otimes n})\rangle$ from \cite[§8, Section 4]{quillen} (see also \cite[Theorem 5.1]{bernardara2009semiorthogonal}) (indeed it follows from this that $\PPic(P_{-})\simeq \PPic_P \simeq \prod_{k=0}^n \PPic_{A^{\otimes k}}$ as sheaves on $X$).
\end{rmk}
As a sample corollary we have the following restricted injectivity announced in the introduction:
\begin{cor}\label{cor:restinj}
    Let $X$ be a qcqs scheme. The composite $$H^2_\et(X,\Gm)\to \Br^\der(X)\to \Pic(\Mot_X)$$ is injective. 
\end{cor}
Let us now prove full injectivity, which is our main theorem. As a warm-up, we begin with the affine case: 
\begin{cor}\label{cor:injaff}
    Let $R$ be a commutative ring. The map $\Br^\der(R)\to \Pic(\Mot_R)$ is injective. 
\end{cor}
Before giving a proof, we point out the following, which will allow us to use \cite[Theorem B.15]{TvdB2} without restricting to noetherian rings\footnote{In fact, from the noetherian case there are many different ways to deduce the genral case of the above corollary.}: 
\begin{lm}\label{lm:Tvdbnonoeth}
    The conclusion of \cite[Theorem B.5]{TvdB2} holds without the noetherian assumption, that is, let $R$ be a commutative ring and $A$ a derived Azumaya algebra over $R$. If $U_R(A)_\mathbb Q\simeq U_R(R)_\mathbb Q$ then $A$ is Morita equivalent to a classical Azumaya algebra. 
\end{lm}
\begin{proof}
    The only step in the proof of \cite[Theorem B.15]{TvdB2} where the authors use that $R$ is noetherian is to prove \cite[Lemma B.16]{TvdB2}, which states that for a perfect $R$-module $P$, if $\mathrm{rank}(P)\neq 0$, then $P$ generates $D(R)$.  One can in fact \emph{deduce} the general case from the noetherian case here by observing that such a $P$ is basechanged from a finite type and hence noetherian commutative ring $R'$ with a map to $R$, where $P'$ must also have nontrivial rank, and therefore be a generator of $D(R')$. But now generators basechange to generators, so we are done. 
\end{proof}

\begin{proof}[Proof of \Cref{cor:injaff}]
Let $A$ be an Azumaya algebra over $R$. Suppose $\U^R(A)\simeq \U^R(R)$. Then, by \Cref{lm:Tvdbnonoeth} (i.e. the non-noetherian version of \cite[Theorem B.15]{TvdB2}), $A$ is equivalent to a \emph{classical} Azumaya algebra (this is where we use that $R$ is noetherian). Hence, its image in $H^1_\et(\Spec(R),\Z)$ is $0$, so that it belongs to the subgroup $H^2_\et(\Spec(R),\Gm)\subset \Br^\der(R)$ (in fact even the torsion subgroup thereof). 

Therefore we may conclude by the previous corollary. 
\end{proof}
In the proof of \Cref{cor:injaff}, affineness is used only very mildly: tracing through \cite[Appendix B]{TvdB2}, one finds that it is used in two places: firstly, to apply the result of Gabber that guarantees that torsion classes in $H^2_\et(\Spec(R),\Gm)$ are represented by classical Azumaya algebras, and second, to prove \cite[Lemma B.16]{TvdB2}. The former is known not to be true in full generality, but the authors prove in \cite[Theorem B.12]{TvdB2} that no matter what, if the rank of an Azumaya algebra $A$ is everywhere nontrivial, then this Azumaya algebra is \emph{torsion} in the Brauer group, with no affineness assumption (and this is where they \emph{then} use affineness to deduce classicality, but we will not actually need this). It turns out that being torsion is enough, essentially since $H^1_\et(X,\Z)$ is torsion-free. 

The latter can also be proved in general, by taking into account $\QCoh(X)$-linearity, and we do so below (so that our \Cref{lm:Tvdbnonoeth} gets a second proof): 
    \begin{lm}\label{lm:thomason}
        Let $X$ be a qcqs scheme and $P$ a perfect quasicoherent sheaf on $X$ with nowhere vanishing rank. Then the localizing ideal generated by $P$ in $\QCoh(X)$ is the whole of $\QCoh(X)$.  
    \end{lm}
    \begin{proof}
      Since $P$ has nowhere vanishing rank, its \emph{support} (the subset of $|X|$ given by those $x$ such that $P_x\neq 0$) is the whole of $|X|$. It follows by Thomason's classification theorem \cite[Theorem 3.14]{thomason} that the thick $\otimes$-ideal generated by $P$ is $\Perf(X)$, from which the claim follows. 
    \end{proof}
We therefore obtain our main theorem as stated in the introduction:
\begin{thm}
    Let $X$ be a qcqs scheme. The map $\Br^\der(X)\to \Pic(\Mot_X)$ is injective.
\end{thm}
\begin{proof}
    The proof was essentially explained before: let $A$ be an Azumaya algebra with $\U^X(A)\simeq \U^X(X)$. 

    The proof of \cite[Theorem B.15]{TvdB2} works with no modification until the point where one concludes that $A$ is Morita equivalent to $B$ where $B$ is an Azumaya algebra with nowhere vanishing rank thanks to the above \Cref{lm:thomason}. 

    Therefore, \cite[Theorem B.11]{TvdB2} applies and we find that $B$ (or equivalently $A$) is torsion in the Brauer group. Since $H^1_\et(X,\Z)$ is torsion-free, it follows that the image of $A$ therein is $0$, and so the Brauer class of $A$ lives in the subgroup $H^2_\et(X,\Gm)$. There, we may conclude by \Cref{cor:restinj} that $A$ is Morita equivalent to $X$. 
\end{proof}

We conclude with the following over-optimistic questions. We raise them nonetheless as an invitation to investigate questions of this nature:
\begin{ques}
    On top of the elements coming from $\Br(X)$, $\Pic(\Mot_X)$ also contains suspensions $\Sigma^n \U^X(X), n \in\Z$. Do these elements generate $\Pic(\Mot_X)$ ? Do they at least generate the picard group of the thick subcategory of $\Mot_X$ generated by motives of smooth proper $X$-schemes ? 
\end{ques}
\begin{ques}
    If we change the base $\Perf(X)/\QCoh(X)$ to something less geometric, where we do not necessarily have access to a local analysis of $\BBr$, what can we say about injectivity ? That is, let $\mathcal{V}\in\CAlg(\PrL_\st)$, and let $\Br(\mathcal{V}):= \Pic(\Mod_\mathcal{V}(\PrL_\st))$. How often is the natural map $\U^\mathcal{V}: \Br(\mathcal{V})\to \Pic(\Mot_\mathcal{V})$ injective ? 
\end{ques}
The above map is provably not injective in full generality - one can for example find counterexamples with $\mathcal{V}= \Sh(S^1,\Sp)$.
\begin{ques}
    How often is the map $\Br(X) \to GL_1(K_0^{(2)}(X))$ injective ?
\end{ques}
\begin{rmk}
    Let $\Mot_X^{\mathrm{loc.cell.},\dbl}\subset \Mot_X$ denote the full subcategory spanned by those $X$-linear motives $M$ that are dualizable and such that there is an étale cover $U\to X$ where $M_{\mid U}$ belongs to the thick subcategory generated by $\one_U$. 

    One can adapt the methods of the companion paper \cite{companion} to prove that $H^2_\et(X,\Gm)\subset\Br(X)\to GL_1(K_0(\Mot_X^{\mathrm{loc.cell.},\dbl}))$ is injective, at least if $X$ is regular, or with no assumption if one considers additive motives instead. The target here is a variant of $K_0^{(2)}(X)$, but it is a bit too small and for example does not even contain all smooth projective $X$-schemes, which is why we do not go into the details of this statement. 
\end{rmk}
\bibliographystyle{alpha}
\bibliography{Biblio.bib}

\end{document}